\documentclass[12pt]{article}
\usepackage[dvips]{color}
\usepackage{geometry}
\usepackage{graphicx}
\usepackage{amsmath,amsthm}
\usepackage{amssymb}
\usepackage{multicol}
\usepackage{ulem}
\usepackage{amsfonts}
\usepackage{mathrsfs}
\usepackage[pdfstartview=FitH]{hyperref}
 \newtheorem{thm}{Theorem}[section]
 \newtheorem{cor}[thm]{Corollary}
 \newtheorem{lem}[thm]{Lemma}
 \newtheorem{prop}[thm]{Proposition}
 
 \theoremstyle{definition}
 \newtheorem{dfn}[thm]{Definition}

 \newtheorem{nott}{Notation}
 \newtheorem{rem}{Remark}
  \DeclareMathAlphabet{\mathsfsl}{OT1}{cmss}{m}{sl}

  \newcommand{\RE}{\mathrm{Re}}

 \newcommand{\Rnum}{\mathbb{R}}
 \newcommand{\Cnum}{\mathbb{C}}
 
 \newcommand{\Nnum}{\mathbb{N}}
 \newcommand{\mi}{\mathrm{i}}
 \newcommand{\dif}{\mathrm{d}}
  
 \newcommand{\tensor}[1]{\mathsf{#1}}
 \newcommand{\abs}[1]{\left\vert#1\right\vert}
 \newcommand{\set}[1]{\left\{#1\right\}}
 \newcommand{\norm}[1]{\left\Vert#1\right\Vert}
 \newcommand{\innp}[1]{\langle {#1}\rangle}
 \newcommand{\bilf}[1]{\big( {#1}\big )_{\theta}}
 
\allowdisplaybreaks

\title{On the 1-dimensional complex Ornstein-Uhlenbeck operator}
\author{\rm\small
\noindent CHEN Yong\\
\noindent \footnotesize School of Mathematics and Computing Science, Hunan
University of Science and Technology,\\
\noindent \footnotesize Xiangtan, Hunan, 411201,
P.R.China. chenyong77@gmail.com\\
}
\date{}
\begin{document}
\maketitle
\maketitle \noindent {\bf Abstract } \\
We show that for any fixed $\theta\in(-\frac{\pi}{2},\,0)\cup (0,\,\frac{\pi}{2})$, the 1-dimensional complex Ornstein-Uhlenbeck operator
 \begin{equation*} 
\tilde{\mathcal{L}}_{\theta}= 4\cos\theta \frac{\partial^2}{\partial z\partial \bar{z}}-e^{\mi\theta} z \frac{\partial}{\partial z}-e^{-\mi\theta}\bar{z} \frac{\partial}{\partial \bar{z}},
\end{equation*} 
is a normal (but nonsymmetric) diffusion operator.
\maketitle
\section{Introduction}
This paper is a sequel of the previous paper \cite{cl}, in which the aim was to obtain the eigenfunctions of 1-dimensional complex Ornstein-Uhlenbeck operator \cite[Theorem~2.6]{cl}
\begin{equation}\label{ouopert}
\tilde{\mathcal{L}}_{\theta}= 4\cos\theta \frac{\partial^2}{\partial z\partial \bar{z}}-e^{\mi\theta} z \frac{\partial}{\partial z}-e^{-\mi\theta}\bar{z} \frac{\partial}{\partial \bar{z}},
\end{equation} where $\theta\in(-\frac{\pi}{2},\,\frac{\pi}{2})$ is fixed and $\frac{\partial f}{\partial z}=\frac12 (\frac{\partial f}{\partial x}-\mi \frac{\partial f}{\partial y}), \frac{\partial f}{\partial \bar{z}}=\frac12 (\frac{\partial f}{\partial x}+\mi \frac{\partial f}{\partial y})$ are
the Wirtinger derivatives of $f$ at point $z=x+\mi y$ with $x,y\in\Rnum$.
They show that the eigenfunctions are the complex Hermite polynomials and form an orthonormal basis of $L^2(\gamma)$ where $\dif \gamma= \frac{1}{2 \pi}e^{-\frac{x^2+y^2}{2}}\dif x\dif y$ (see Proposition~\ref{lm22} below).

In this paper, we will firstly show that $\tilde{\mathcal{L}}_{\theta}$ can be realized as an unbounded normal operator (see \cite[p368]{ru}) in $L^2(\gamma)$ but nonsymmetric when $\theta\neq 0$. Secondly, we extend the known fact about the 1-dimensional real symmetric diffusion operator \cite{bd,str0,str} to the complex case. Precisely stated, we present the explicit expression of $\tilde{\mathcal{L}}_{\theta}$ in $L^2(\gamma)$ (see Theorem~\ref{exp_normal}) and show that it is a normal diffusion operator (see Theorem~\ref{main_thm}).

This article is organized as follows. Section~\ref{sec2} provides necessary information of complex Hermite polynomials. Section~\ref{normty} contains the proof of the normality of the complex Ornstein-Uhlenbeck semigroup. Section~\ref{sec4} contains the main results on the explicit expression of $\tilde{\mathcal{L}}_{\theta}$ and the property of the normal diffusion operator. Finally, some necessary approximation of identity and N-representation theorem are listed in Appendix.
\section{Preliminaries}\label{sec2}
\begin{dfn}{\bf (Definition of the complex Hermite polynomials \cite[Definition~2.4]{cl})}\label{ijldf}
We call $\partial:=\frac{\partial }{\partial z}$ and $ \bar{\partial}:=\frac{\partial }{\partial \bar{z}}$ the complex annihilation operators.
Let $m,n\in \Nnum$. We define the sequence on $\Cnum$ (or say: $\Rnum^2$)
\begin{align*}
  J_{0,0}(z)&=1,\nonumber\\
  J_{m,n}(z)&=\sqrt{\frac{{2}^{m+n}}{m!n!}}(\partial^*)^m(\bar{\partial}^*)^n 1,\label{itldefn}
\end{align*} where $ (\partial^*\phi)(z)=-\frac{\partial}{\partial \bar{z}}\phi(z)+\frac{z }{2}\phi(z),\, (\bar{\partial}^*\phi)(z)=-\frac{\partial}{\partial {z}}\phi(z)+\frac{\bar{z}}{2 }\phi(z)$ for $\phi\in C^1_{\uparrow}(\Rnum^2)$ (see Definition~\ref{dfn5_1}) are the adjoint of the operators $\partial,\,\bar{\partial} $ in $L^2(\gamma)$ respectively (the complex creation operator).
\end{dfn}
 In \cite[Theorem 2.7, Corollary 2.8]{cl},  the authors show that $ J_{m,n}(z)$ satisfies that:
\begin{prop}\label{lm22}
The complex Hermite polynomials $\set{J_{m,n}(z):\,m,n\in \Nnum}$ form an orthonormal basis of $L^2(\gamma)$ where $\dif \gamma= \frac{1}{2 \pi}e^{-\frac{x^2+y^2}{2}}\dif x\dif y$. Thus, every function $f$ in $L^2(\gamma)$ has a unique series expression
  \begin{equation*}\label{amn0}
    f=\sum_{m=0}^{\infty}\sum_{n=0}^{\infty}b_{m,n} J_{m,n}(z),
  \end{equation*}
  where the coefficients $b_{m,n}$ are given by
  \begin{equation*}\label{amn}
  b_{m,n}=\innp{f,\,J_{m,n}}=\int_{\Rnum^2}f \overline{ J_{m,n}(z)}\dif \gamma.
\end{equation*}
Moreover, for any $\theta\in(-\frac{\pi}{2},\,\frac{\pi}{2})$ and each $m,n\in \Nnum$,
\begin{equation}\label{atheta}
  \tilde{\mathcal{L}}_{\theta} J_{m,n}(z)=-[(m+n)\cos \theta +\mi (m-n)\sin\theta]J_{m,n}(z).
\end{equation}
\end{prop}

The real Hermite polynomials are defined by the formula\footnote{Note that $H_n(x)=\frac{(-1)^n}{n!}  e^{x^2/2}\frac{\dif^n}{\dif x^n}e^{-x^2/2}$ in \cite{Nua,shg} and $H_n(x)= (-1)^n e^{x^2/2}\frac{\dif^n}{\dif x^n}e^{-x^2/2}$ in \cite{cl,guo}, here we use the definition in \cite{str}.} $$H_n(x)=\frac{(-1)^n}{\sqrt{n!}} e^{x^2/2}\frac{\dif^n}{\dif x^n}e^{-x^2/2},\,n=1,2,\dots.$$
The following property gives the fundamental relation between the real and the complex Hermite polynomials \cite[Corollary 2.8]{cl}.
\begin{prop}\label{2dim2}
Let $z=x + \mi y$ with $x,y\in \Rnum$. Then the real and the complex Hermite polynomials satisfy that
\begin{equation}\label{h2j}
    \begin{array}{ll}
      J_{m,l-m}(z) = \sum\limits_{k=0}^{l}{\mi^{l-k}}\sqrt{\frac{k!(l-k)!}{ 2^l m!(l-m)!}}\sum_{r+s=k}{m \choose r}{l-m \choose s}(-1)^{l-m-s} H_k(x)H_{l-k}(y),\\
        H_k(x)H_{l-k}(y) =  \mi^{l-k} \sum\limits_{m=0}^l \sqrt{\frac{m!(l-m)!}{2^lk!(l-k)!}} \sum_{r+s=m}{k \choose r}{l-k \choose s}(-1)^{s} J_{m,l-m}(z).
    \end{array}
 \end{equation}
Thus, both the class $\set{ J_{k,l}(z):\,k+l=n}$ and the class $\set{ H_k(x)H_{l}(y):\,k+l=n }$ generate the same linear subspace of $L^2(\gamma)$.
\end{prop}

\section{The normality of the complex Ornstein-Uhlenbeck semigroup}\label{normty}
In \cite{cl}, the following complex Ornstein-Uhlenbeck process $\set{Z_t}$ is considered:
\begin{equation}\label{cp}
\left\{
      \begin{array}{ll}
      \dif Z_t=-e^{\mi \theta} Z_t\dif t+ \sqrt{2\cos\theta}\dif \zeta_t,\quad t\ge 0,  \\
      Z_0=x \in \Cnum ,
      \end{array}
\right.
\end{equation}
where $\theta\in (-\frac{\pi}{2},\,\frac{\pi}{2})$, and $\zeta_t=B_1(t)+\mi B_2(t)$ is a complex Brownian motion. Solving for $Z$ gives
\begin{equation}\label{solv}
   Z_t=e^{- (\cos \theta+\mi \sin \theta)t}\big( x+ \sqrt{2\cos\theta} \int_{0}^t  e ^{ (\cos \theta+\mi \sin \theta)s}\, \dif \zeta_s \big).
\end{equation}
Thus, the associated Ornstein-Uhlenbeck semigroup of Eq.(\ref{cp}) has the following explicit representation, due to Kolmogorov, for each $\varphi\in C_b(\Rnum^2)$ (the space of all continuous and bounded complex-valued functions on $\Rnum^2$),
\begin{align}
     P_t \varphi (x)&=E_x[ \varphi(Z_t) ]\label{compp}\\
     &=\frac{1}{2\pi (1-e^{-2 t\cos\theta})}\int_{\Rnum^2}\,e^{-\frac{\abs{y}^2}{2(1-e^{-2 t\cos\theta})}}\varphi(e^{- (\cos \theta+\mi \sin \theta)t}x-y)\,\dif y_1\dif y_2,\label{compp2}
\end{align} where $y=y_1+\mi y_2$ and $x,y\in \Cnum$ and we write a function $\varphi(y_1,y_2)$ of the two real variables $y_1$ and $y_2$ as $\varphi(y)$ of the complex argument $y_1+\mi y_2$ (i.e.,
we use the complex representation of $\Rnum^2$ in (\ref{compp}-\ref{compp2})). The change of variable formula yields the following  Mehler formula \cite[p584]{cl}.
\begin{prop}{\bf(Mehler formula)}\label{lm209}
For each $\varphi\in C_b(\Rnum^2)$,
\begin{equation}\label{ousemi}
  P_t\varphi(x)= \int_{\Cnum}\,\varphi(e^{- (\cos \theta+\mi \sin \theta)t} x +\sqrt{1-e^{-2t\cos \theta}}y)\,\dif \gamma(y),
\end{equation}
where
\begin{equation}\label{meas}
\dif\gamma(y) = \frac{1}{2\pi }\exp\set{-\frac{(y_1^2+y_2^2)}{2}}\dif y_1\dif y_2
\end{equation}
\end{prop}
Similar to the real Ornstein-Uhlenbeck semigroup \cite[Proposition~2.3]{shg}, using the rotation invariant of the measure $\gamma$ and Lebesgue's dominated convergence theorem, it follows from Proposition~\ref{lm209} that $\gamma$ is the unique invariant measure of $P_t$. In detail, for each $\varphi\in C_b(\Rnum^2)$,
\begin{equation}\label{inv}
  \int_{\Cnum}  P_t\varphi(x)\,\dif \gamma(x)= \int_{\Cnum}  \varphi(x)\,\dif \gamma(x)
\end{equation}
and
\begin{equation}\label{uniq}
  \lim_{t\to \infty}  P_t\varphi(x)= \int_{\Cnum}  \varphi(y)\,\dif \gamma(y),\quad \forall x\in \Cnum.
\end{equation}

Denote the associated transition probabilities on $\Cnum$ as $P_t(x,A)=P_t\mathbf{1}_A(x)$ for each $A\in \mathcal{B}(\Rnum^2)$.
Along the same line of the real case \cite{bd,shg,str}, for each $p\ge 1$, it follows from Jensen's inequality that for each $\varphi\in C_b(\Rnum^2)$,
\begin{align*}
  \norm{P_t\varphi}^p_{L^p(\gamma)}  & =  \int_{\Cnum}\abs{ P_t\varphi(x)}^p\,\dif \gamma(x)= \int_{\Cnum}\abs{\int_{\Cnum }\varphi(y) P_t(x,\,\dif\, y) }^p\,\dif \gamma(x)\\
  &\le \int_{\Cnum}\,\dif \gamma(x)\int_{\Cnum }\abs{\varphi(y) }^p P_t(x,\,\dif\, y)\\
  &= \int_{\Cnum}  \abs{\varphi }^p (x )\,\dif \gamma(x)\quad(\text {by (\ref{inv})})\\
  &=\norm{ \varphi}^p_{L^p(\gamma)}.
\end{align*}
It follows from the B.L.T. theorem \cite[p9]{rs} that $\set{P_t}_{t\ge 0}$ can be uniquely extended to a strong continuous contraction semigroup $\set{T_t^p}_{t\ge 0}$ on $L^p(\gamma)$ for each $p\ge 1$ \footnote{Namely, $T_t^p$ is the closure (see \cite[p250]{rs}) in $L^p(\gamma)$ of the operator $P_t$.}. Let $\mathcal{A}_p$ be the (infinitesimal) generator, then $\mathcal{A}_p$ is closed and $\overline{D(\mathcal{A}_p)}=L^p(\gamma)$ (i.e., densely defined) \cite{Pazy, rs}.

\begin{lem}\label{l2110}
Suppose $\vec{Y}=(y_1,\dots,y_n),\,\vec{Z}=(z_1,\dots,z_n)\in \Cnum^n$ and
$  \vec{Y}=\tensor{M}\vec{Z}$,
where $\tensor{M}=(M_{ij})$ is an $n$-by-$n$ unitary matrix over the field $\Cnum$.
If $z_i,i=1\dots,n$ are independent, each being centered complex normal such that $E\abs{z_i}^2=\sigma^2$, then $y_i,i=1\dots,n$ are also independent, each being centered complex normal such that $E\abs{y_i}^2=\sigma^2$.
\end{lem}
\begin{proof} It follows from \cite[Theorem 1.1]{ito} that $y_i,i=1\dots,n$ are centered complex normal. In addition, we have that
\begin{align*}
  E[y_i\bar{y}_j] & = \sum_{k,l}M_ik E[z_k\bar{z}_l] \bar{M}_{jl} =\sigma^2  \sum_{k} M_{ik}\bar{M}_{jk}=\sigma^2 \delta_{ij},
\end{align*} where $\delta_{ij}$ is the Kronecker delta. Thus, $y_i,i=1\dots,n$ have independent identical distributions with variance $\sigma^2$.
\end{proof}
\begin{prop}
For each $\theta\in (-\frac{\pi}{2},\,\frac{\pi}{2})$,
denote the semigroup $P_t$ depending on $\theta$ in (\ref{compp}) by $P_t^{\theta}$, then for each $\phi \in C_b(\Rnum^2)$
\begin{equation}\label{nonom}
  P_t^{\theta} (P^{\theta}_t)^*\phi(x)=(P^{\theta}_t)^* P_t^{\theta}\phi(x),
\end{equation}
where $(P^{\theta}_t)^*$ is the adjoint operator of $P^{\theta}_t$ in $L^2(\gamma)$. Furthermore, when restricted on $C_b(\Rnum^2)$, $(P^{\theta}_t)^*=P^{-\theta}_t$.
\end{prop}
\begin{proof}
 Set $\alpha=e^{\mi\theta}$.
 For each $\phi,\psi\in C_b(\Rnum^2)$ and $t\ge0$, we have that
 \begin{align}
    \innp{P^{\theta}_t\phi,\,\psi} & =  \int_{\Cnum}\,\bar{\psi}(z_1)\,\dif \gamma(z_1) \int_{\Cnum}\,\phi(e^{- \alpha t} z_1 +\sqrt{1-e^{-2t\RE{\alpha}}}z_2)\,\dif \gamma(z_2)\nonumber\\
     &= \int_{\Cnum}\, {\phi}(y_1)\,\dif \gamma(y_1) \int_{\Cnum}\,\overline{\psi(e^{- \bar{\alpha}t} y_1 -\sqrt{1-e^{-2t\RE{\alpha}}}y_2)}\,\dif \gamma(y_2)\label{l211}\\
     &=\int_{\Cnum}\, {\phi}(y_1)\,\dif \gamma(y_1) \int_{\Cnum}\,\overline{\psi(e^{- \bar{\alpha}t} y_1 +\sqrt{1-e^{-2t\RE{\alpha}}}y_2)}\,\dif \gamma(y_2)\label{l21100}\\
     &=\innp{\phi,\,P^{-\theta}_t \psi}\nonumber,
 \end{align} where (\ref{l211}) is deduced from Lemma~\ref{l2110} by taking $n=2$ and
 \begin{equation*}
 \tensor{M}=\begin{bmatrix} e^{-\alpha t} & \sqrt{1-e^{-2t\RE{\alpha}}}\\ -\sqrt{1-e^{-2t\RE{\alpha}}} & e^{-\bar{\alpha} t}  \end{bmatrix},
\end{equation*}
and (\ref{l21100}) is deduced from the rotation invariant of the measure $\gamma$.
Therefore, the adjoint operator of $P^{\theta}_t$ in $L^2(\gamma)$ satisfies that $(P^{\theta}_t)^*=P^{-\theta}_t$ when restricted on $C_b(\Rnum^2)$ for each $\theta\in (-\frac{\pi}{2},\,\frac{\pi}{2})$. Thus, for each $\phi \in C_b(\Rnum^2)$,
\begin{align}
  P_t^{\theta} (P^{\theta}_t)^*\phi(x) & = \int_{\Cnum^2} \phi\big(e^{- \bar{\alpha} t}(e^{-  {\alpha} t} x +\sqrt{1-e^{-2t\RE{\alpha}}}z_1)+ \sqrt{1-e^{-2t\RE{\alpha}}}z_2\big)\,  \dif \gamma(z_1)\dif \gamma(z_2)\nonumber\\
   & = \int_{\Cnum } \phi\big(e^{- 2 t\RE{\alpha}}  x +\sqrt{1-e^{-4t\RE{\alpha}}}z \big)\,  \dif \gamma(z)\label{convpd}\\
   &=\int_{\Cnum^2} \phi\big(e^{-{\alpha} t}(e^{-  \bar{\alpha} t} x +\sqrt{1-e^{-2t\RE{\alpha}}}z_1)+ \sqrt{1-e^{-2t\RE{\alpha}}}z_2\big)\,  \dif \gamma(z_1)\dif \gamma(z_2)\nonumber\\
   &= (P^{\theta}_t)^* P_t^{\theta}\phi(x),\nonumber
\end{align}
where (\ref{convpd}) is deduced from the well-known fact that if $Z_1,Z_2$ are two independent standard complex normal random variables, then $e^{- \bar{\alpha} t} \sqrt{1-e^{-2t\RE{\alpha}}}Z_1+ \sqrt{1-e^{-2t\RE{\alpha}}}Z_2 $ and $\sqrt{1-e^{-4t\RE{\alpha}}}Z_1$ have the same law \cite[Theorem 1.1]{ito}.
\end{proof}

\begin{thm}\label{lm212}
  $\set{T_t^2}_{t\ge 0}$ is a semigroup of normal operators (see \cite[p382]{ru}) in $L^2(\gamma)$ and thus the generator $\mathcal{A}_2$ is a normal operator in $L^2(\gamma)$.
\end{thm}
\begin{proof}
   Since $ T_t^2 $ is the closure of the contraction operator $P_t$ in $L^2(\gamma)$, it follows from (c) of Theorem~VIII.1 in \cite[p253]{rs} that
   the adjoint operator of $T_t^2$ equals to that of $P_t$. It follows from the density argument that (\ref{nonom}) can be extended to each $\phi\in L^2(\gamma)$, i.e., $$T_t^2(T_t^2)^*\phi=(T_t^2)^* T_t^2\phi,\qquad \forall \phi\in L^2(\gamma).$$
   Thus $\set{T_t^2}_{t\ge 0}$ is a semigroup of normal operators. It follows from \cite[Theorem~13.38]{ru} that the generator $\mathcal{A}_2$ is a normal operator in $L^2(\gamma)$.
\end{proof}
\section{The normal diffusion operators in $\Cnum$}\label{sec4}
The first aim of this section is to show the explicit expression of the generator $\mathcal{A}_2$.
\begin{dfn}\label{cld}
For any $\theta\in(-\frac{\pi}{2},\,\frac{\pi}{2})$, define
   \begin{equation}\label{dom}
     \mathcal{D}(\mathcal{L}_{\theta})=\set{f\in L^2(\gamma),\,\sum_{m=0}^{\infty}\sum_{n=0}^{\infty} (m^2+n^2+2mn\cos2\theta)\abs{\innp{f,\,J_{m,n}}}^2<\infty  }
   \end{equation}
   and
   \begin{equation}\label{opl}
     \mathcal{L}_{\theta}f=-\sum_{m=0}^{\infty}\sum_{n=0}^{\infty}[(m+n)\cos \theta +\mi (m-n)\sin\theta]\innp{f,\,J_{m,n}} J_{m,n}(z).
   \end{equation}
\end{dfn}
\begin{thm}\label{pop27}
   If $\phi\in C^2_{\uparrow}(\Rnum^2)$, then  $\phi\in \mathcal{D}( \mathcal{L}_{\theta})$ and
   \begin{equation}\label{deri}
     \mathcal{L}_{\theta} \phi=[ 4\cos\theta \frac{\partial^2}{\partial z\partial \bar{z}}-e^{\mi\theta} z \frac{\partial}{\partial z}-e^{-\mi\theta}\bar{z} \frac{\partial}{\partial \bar{z}}]\phi.
   \end{equation}
\end{thm}

\begin{thm} \label{exp_normal}
Let $\mathcal{A}_2$ be as in Theorem~\ref{lm212} and $\mathcal{L}_{\theta}$ be as in Definition~\ref{cld}.
For any $\theta\in(-\frac{\pi}{2},\,\frac{\pi}{2})$, $\mathcal{A}_2=\mathcal{L}_{\theta}$, i.e., $
   D(\mathcal{A}_2)=\mathcal{D}(\mathcal{L}_{\theta})$ and $\mathcal{A}_2\varphi=\mathcal{L}_{\theta}\varphi$ on $ \mathcal{D}(\mathcal{L}_{\theta})$.
\end{thm}
The second aim of this section is to show that the operator $\mathcal{L}_{\theta}$ defined above satisfies the following theorem,
which is named as the {\bf normal diffusion operator} analogous to the symmetric diffusion operator given by Stroock \cite{bd,str0,str}.
\begin{thm}{\bf (normal diffusion operator)}\label{main_thm}
The densely defined linear closed operator $\mathcal{L}_{\theta}$ defined in Proposition~\ref{cld} is a normal diffusion operator. Namely, it satisfies that:
     \begin{itemize}
        \item[\textup{1)}] $\mathcal{L}_{\theta}$ is a normal operator on $\mathcal{D}(\mathcal{L}_{\theta} )$.
        \item[\textup{2)}] $1\in \mathcal{D}(\mathcal{L}_{\theta} )$ and $\mathcal{L}_{\theta}1=0$.
        \item[\textup{3)}] There exists a linear subspace $
   \mathcal{D}\subset \set{\phi\in \mathcal{D}(\mathcal{L}_{\theta})\cap L^4(\gamma):\,\mathcal{L}_{\theta}\phi\in L^4(\gamma),\,\abs{\phi}^2\in \mathcal{D}(\mathcal{L}_{\theta})} $ such that $\mathrm{graph}(\mathcal{L}_{\theta}|\mathcal{D} ) $ is dense in $\mathrm{graph}(\mathcal{L}_{\theta})$.
        \item[\textup{4)}] For any $\theta\in(-\frac{\pi}{2},\,\frac{\pi}{2})$, define
\begin{equation*}
\bilf{\phi,\,\psi}=\frac{1}{2\cos\theta}[\mathcal{L}_{\theta}(\phi\bar{\psi })- \phi\mathcal{L}_{\theta}(\bar{\psi })-\bar{\psi }\mathcal{L}_{\theta}(\phi)]
\end{equation*}
for $\phi,\,\psi\in \mathcal{D}$. Then $\bilf{\cdot,\,\cdot}:\, \mathcal{D}\times \mathcal{D}\to L^2(\gamma)$ is a non-negative definite bilinear form on the field $\Cnum$.
        \item[\textup{5)}] (Diffusion property)If $\vec{\phi}=(\phi_1,\dots,\phi_n)\in \mathcal{D}^n$ and $F\in C^2_{\uparrow}(\Cnum^{n})$, then $F\circ\vec{\phi}\in \mathcal{D}(\mathcal{L}_{\theta})$ and
\begin{align}\label{ltheta}
  \mathcal{L}_{\theta}(F\circ\vec{\phi})&=\cos\theta  \sum_{i,j=1}^n \bilf{\phi_i,\,\bar{\phi}_j}\frac{\partial^2 F}{\partial z_i\partial{ z_j}}\circ\vec{\phi} +\bilf{\bar{\phi}_i,\,\phi_j}\frac{\partial^2 F}{\partial \bar{z}_i\partial\bar{ z}_j}\circ\vec{\phi} +2\bilf{\phi_i,\,\phi_j}\frac{\partial^2 F}{\partial z_i\partial\bar{ z}_j}\circ\vec{\phi} \nonumber \\ &+\sum_{i=1}^n \mathcal{L}_{\theta}\phi_i\frac{\partial F}{\partial z_i}\circ\vec{\phi}+ \mathcal{L}_{\theta}\bar{\phi}_i\frac{\partial F}{\partial \bar{z}_i}\circ\vec{\phi}.
\end{align}
        \item[\textup{6)}] $\mathcal{L}_{\theta}$ has an extension $\mathcal{A}_1$ to $L^1(\gamma)$ with domain $\mathcal{D}(\mathcal{A}_1)$ such that $$
     \mathcal{D}(\mathcal{L}_{\theta})=\set{\phi\in D(\mathcal{A}_1)\cap L^2(\gamma):\,\mathcal{A}_1\phi\in L^2(\gamma) },$$
     i.e., the closure of $\mathcal{L}_{\theta}$ in $L^1(\gamma)$ is $\mathcal{A}_1$.
     \end{itemize}
\end{thm}
Proofs of Theorem~\ref{pop27}-\ref{main_thm} are presented in Section~\ref{sec_31}.
\subsection{Proofs of Theorems}\label{sec_31}
\begin{prop}\label{cld1}
  Let $\mathcal{L}_{\theta}$ be as in Definition~\ref{cld}. Then $\mathcal{L}_{\theta} $ is closed on $ L^2(\gamma)$.
\end{prop}
\begin{proof}
   Suppose that $f_k\in  \mathcal{D}(\mathcal{L}_{\theta}) $ such that $f_k\to f,\,\mathcal{L}_{\theta}f_k\to g$ in $L^2(\gamma)$, we will show that $f\in  \mathcal{D}(\mathcal{L}_{\theta})$ and $\mathcal{L}_{\theta}f=g$.
   In fact, by Fatou's lemma and Parseval's identity the triangle inequality we have that
   \begin{align*}
    & \sum_{m,n=0}^{\infty} (m^2+n^2+2mn\cos2\theta)\abs{\innp{f,\,J_{m,n}}}^2 \\
    &\le \liminf_{k\to\infty} \sum_{m,n=0}^{\infty} (m^2+n^2+2mn\cos2\theta)\abs{\innp{f_k,\,J_{m,n}}}^2 \\
      & =\liminf_{k\to\infty} \norm{\mathcal{L}_{\theta}f_k }^2  \\
      &=\norm{g}^2<\infty.
   \end{align*}
   Thus $f\in  \mathcal{D}(\mathcal{L}_{\theta})$. In addition, since for each $m,n\ge 0$, as $k\to \infty$, $$\innp{[(m+n)\cos \theta +\mi (m-n)\sin\theta]f_k+g,\,J_{m,n}}\to \innp{[(m+n)\cos \theta +\mi (m-n)\sin\theta]f+g,\,J_{m,n}},$$ it follows from Parseval's identity and Fatou's lemma that
   \begin{align*}
     \norm{\mathcal{L}_{\theta}f-g }^2  & =\sum_{m,n=0}^{\infty}\abs{ \innp{[(m+n)\cos \theta +\mi (m-n)\sin\theta]f+g,\,J_{m,n}}}^2 \\
          &\le\liminf_{k\to \infty}  \sum_{m,n=0}^{\infty}\abs{ \innp{[(m+n)\cos \theta +\mi (m-n)\sin\theta]f_k+g,\,J_{m,n}}}^2\\
          &= \liminf_{k\to \infty} \norm{\mathcal{L}_{\theta}f_k-g }^2 =0.
   \end{align*}
   Thus $\mathcal{L}_{\theta}f=g$.
\end{proof}

\begin{rem} Suppose that $H_{m,n}(x,y)=H_{m }(x )H_{n}(y)$ is the Hermite polynomial of two variables. Then it follows from Proposition~\ref{2dim2} that
\begin{equation*}
   \sum_{m+n=l}\abs{\innp{f,\,J_{m,n}}}^2 = \sum_{m+n=l}\abs{\innp{f,\,H_{m,n}}}^2.
\end{equation*}
Together with $$(m+n)^2\ge m^2+n^2+2mn\cos2\theta=(m+n)^2\cos \theta^2+(m-n)^2\sin\theta^2\ge (m+n)^2\cos \theta^2,$$  we deduce that
   \begin{align}\label{dom000}
     \mathcal{D}(\mathcal{L}_{\theta})&=\set{f\in L^2(\gamma),\,\sum_{m=0}^{\infty}\sum_{n=0}^{\infty} (m+n)^2\abs{\innp{f,\,J_{m,n}}}^2<\infty  }\nonumber\\
     &=\set{f\in L^2(\gamma),\,\sum_{m=0}^{\infty}\sum_{n=0}^{\infty} (m+n)^2\abs{\innp{f,\,H_{m,n}}}^2<\infty  },
   \end{align}
   that is to say, $\mathcal{D}(\mathcal{L}_{\theta})$ is independent to $\theta$. In fact, the right hand side of (\ref{dom000}) is exact the Sobolev weighted space $H^2_{\gamma}$, please refer to \cite{lun} for details.
\end{rem}

\begin{prop}\label{pop25}
   For any $\theta\in(-\frac{\pi}{2},\,\frac{\pi}{2})$, $\mathcal{L}_{\theta}$ is a normal operator on $\mathcal{D}(\mathcal{L}_{\theta} )$ such that $1\in \mathcal{D}(\mathcal{L}_{\theta} )$ and $\mathcal{L}_{\theta}1=0$.
\end{prop}
\begin{proof}
Suppose that $f,g\in \mathcal{D}(\mathcal{L}_{\theta} )$. It follows from Parseval's identity that
\begin{align*}
  \innp{\mathcal{L}_{\theta}f,\,g} & = -\sum_{m,n}^{\infty} [(m+n)\cos \theta +\mi (m-n)\sin\theta]\innp{f,\,J_{m,n}}\overline{\innp{g,\,J_{m,n}}}\\
  & =-\sum_{m,n}^{\infty}\innp{f,\,J_{m,n}}\overline{[(m+n)\cos \theta - \mi (m-n)\sin\theta]\innp{g,\,J_{m,n}}}\\
  & = \innp{f,\,\mathcal{L}_{-\theta}g}.
\end{align*}
Thus, the adjoint operator of $\mathcal{L}_{\theta}$ is $\mathcal{L}^*_{\theta}=\mathcal{L}_{-\theta}$. The equality (\ref{dom000}) implies that $\mathcal{D}(\mathcal{L}_{\theta} )=\mathcal{D}(\mathcal{L}_{-\theta} )=\mathcal{D}(\mathcal{L}^*_{\theta} )$. And for each $f\in \mathcal{D}(\mathcal{L}_{\theta} )$ such that $\mathcal{L}_{\theta}f\in \mathcal{D}(\mathcal{L}^*_{\theta})$, we have that
\begin{align*}
 \mathcal{L}^*_{\theta} \mathcal{L}_{\theta}f  =\sum_{m,n}  (m^2+n^2+2mn\cos2\theta){\innp{f,\,J_{m,n}}}J_{m,n}(z)=\mathcal{L}_{\theta}\mathcal{L}^*_{\theta}f.
\end{align*}
Therefore, $\mathcal{L}_{\theta}$ is a normal operator on $\mathcal{D}(\mathcal{L}_{\theta} )$. $1\in \mathcal{D}(\mathcal{L}_{\theta} )$ and $\mathcal{L}_{\theta}1=0$ is trivial.
\end{proof}
\begin{prop}\label{pp26}
Denote by $\mathcal{D}=\mathrm{span}\set{J_{m,n},\,m,n\ge 0}$ the linear span (also called the linear hull) of complex Hermite polynomials. Then
\begin{equation}\label{sp}
   \mathcal{D}\subset \set{\phi\in \mathcal{D}(\mathcal{L}_{\theta})\cap L^4(\gamma):\,\mathcal{L}_{\theta}\phi\in L^4(\gamma),\,\abs{\phi}^2\in \mathcal{D}(\mathcal{L}_{\theta})}
\end{equation}
and $\mathrm{graph}(\mathcal{L}_{\theta}|\mathcal{D} ) $ is dense in  $\mathrm{graph}(\mathcal{L}_{\theta}) $.
\end{prop}
\begin{proof}The equality \cite[Theorem 2.5]{cl}
\begin{equation*}
    J_{m,n}(z)=(m!n!2^{m+n})^{-\frac12}\sum_{r=0}^{m\wedge n}(-1)^r r!2^{r}{m \choose r}{n \choose r}z^{m-r}\bar{z}^{n-r},\quad \forall m,n\ge \Nnum
\end{equation*}
and the equality \cite[Corollary 2.8]{cl}
   \begin{equation*}
      z^m\bar{z}^n= \sum_{k=0}^{m\wedge n} {m \choose k}{n \choose k}k!\sqrt{(m-k)!(n-k)!2^{m+n}} J_{m-k,n-k}(z),\quad \forall m,n\ge \Nnum
   \end{equation*}
imply that $\mathcal{D}=\mathrm{span}\set{J_{m,n},\,m,n\ge 0}=\mathrm{span}\set{ z^m\bar{z}^n,\,m,n\ge 0}$.
But $f(z)=z^m\bar{z}^n$ belonging to the right hand side of (\ref{sp}) is trivial. Thus (\ref{sp}) holds.

Since $\mathcal{D}$ is a dense subset of $L^2(\gamma)$ (see Proposition~\ref{lm22}), $\mathcal{D}$ is dense in $ \mathcal{D}(\mathcal{L}_{\theta})$. Note that  $\mathcal{L}_{\theta}$ is a closed operator, we get that  $\mathrm{graph}(\mathcal{L}_{\theta}|\mathcal{D} ) $ is dense in  $\mathrm{graph}(\mathcal{L}_{\theta}) $.
\end{proof}

\noindent{\it Proof of Theorem~\ref{pop27}.\,}
First, it follows from Proposition~\ref{lm22} that (\ref{deri}) holds when $\phi\in \mathcal{D}=\mathrm{span}\set{J_{m,n},\,m,n\ge 0}$.

Second, suppose that $\phi\in L^2(\gamma)$ satisfies that the sequence $a_{m,n}=\innp{\phi,\,H_m(x)H_n(y)}$ is rapidly decreasing, then we will show that $\phi\in \mathcal{D}( \mathcal{L}_{\theta})$ and (\ref{deri}) is satisfied.
In fact, it follows from Proposition~\ref{tm410} that the Hermite expansion $\phi(z)=\sum_{m,n=0}^{\infty}a_{m,n} H_m(x)H_n(y)$ satisfies that
   \begin{equation*}
     \norm {x^{k_1} y^{k_2}\frac{\partial^{p_1+p_2}}{\partial x^{p_1}\partial y^{p_2}}(\phi-\phi_l)}_{L^2(\gamma)} \to 0 {\text{ as $l\to \infty$,\quad}} \forall k_1,k_2,\,p_1,p_2\in\Nnum,
   \end{equation*}
   where $\phi_l=\sum_{{m+n}\le l}a_{m,n} H_m(x)H_n(y)$.
Thus,
   \begin{equation*}
     \norm {z^{k_1} \bar{z}^{k_2}\frac{\partial^{p_1+p_2}}{\partial z^{p_1}\partial \bar{z}^{p_2}}(\phi-\phi_l)}_{L^2(\gamma)} \to 0 {\text{ as $l\to \infty$,\quad}} \forall k_1,k_2,\,p_1,p_2\in\Nnum.
   \end{equation*}
It follows from Proposition~\ref{2dim2} that $\sum_{{m+n}\le l}a_{m,n} H_m(x)H_n(y)=\sum_{{m+n}\le l}b_{m,n} J_{m,n}(z)$. Thus, as $l\to \infty$, we have that in $L^2(\gamma)$, $\phi_l\to \phi$ and
   \begin{align*}
      \mathcal{L}_{\theta} \phi_l &=-\sum_{{m+n}\le l}[(m+n)\cos \theta +\mi (m-n)\sin\theta] b_{m,n}J_{m,n}(z) \\
      &=[ 4\cos\theta \frac{\partial^2}{\partial z\partial \bar{z}}-e^{\mi\theta} z \frac{\partial}{\partial z}-e^{-\mi\theta}\bar{z} \frac{\partial}{\partial \bar{z}}]\phi_l \\
      & \to [ 4\cos\theta \frac{\partial^2}{\partial z\partial \bar{z}}-e^{\mi\theta} z \frac{\partial}{\partial z}-e^{-\mi\theta}\bar{z} \frac{\partial}{\partial \bar{z}}]\phi.
   \end{align*}
Since $ \mathcal{L}_{\theta}$ is closed, we have that $\phi\in \mathcal{D}( \mathcal{L}_{\theta})$ and (\ref{deri}) is satisfied.

Finally, it follows from Proposition~\ref{pp47} that if $\phi\in C^2_{\uparrow}(\Rnum^2)$ then there exists an approximation of identity $B_{\epsilon}\phi\in C_c^{\infty}(\Rnum^2)$ such that for all $p_1+p_2\le 2$ and $k_1,k_2\ge 0$, $x^{k_1}y^{k_2}\frac{\partial^{p_1+p_2}}{\partial x^{p_1}\partial y^{p_2}}( B_{\epsilon}\phi)\to x^{k_1}y^{k_2}\frac{\partial^{p_1+p_2}}{\partial x^{p_1}\partial y^{p_2}} \phi $ in $L^2(\gamma)$ as $\epsilon\to 0$. In addition, it follows from Proposition~\ref{tm57} that the sequence $\innp{B_{\epsilon}\phi,\, H_{m}(x)H_n(y)}$ is rapidly decreasing.
Thus, as $\epsilon\to 0$, we have that in $L^2(\gamma)$, $B_{\epsilon}\phi\to \phi$ and
\begin{align*}
  \mathcal{L}_{\theta} (B_{\epsilon}\phi) & =  [ 4\cos\theta \frac{\partial^2}{\partial z\partial \bar{z}}-e^{\mi\theta} z \frac{\partial}{\partial z}-e^{-\mi\theta}\bar{z} \frac{\partial}{\partial \bar{z}}] B_{\epsilon}\phi\\
   & \to [ 4\cos\theta \frac{\partial^2}{\partial z\partial \bar{z}}-e^{\mi\theta} z \frac{\partial}{\partial z}-e^{-\mi\theta}\bar{z} \frac{\partial}{\partial \bar{z}}] \phi.
\end{align*}
Since $ \mathcal{L}_{\theta}$ is closed, we have that $\phi\in \mathcal{D}( \mathcal{L}_{\theta})$ and (\ref{deri}) is satisfied.
 {\hfill\large{$\Box$}}
\\

\noindent{\it Proof of Theorem~\ref{exp_normal}.\,}
First, it follows from the density argument (see Proposition~\ref{pp47}) and Lebesgue's dominated convergence theorem that the Mehler formula (\ref{ousemi}) is still valid for $\varphi\in C^0_{\uparrow}(R^2)$, i.e.,
\begin{equation*}
T_t^2\varphi (x)= \int_{\Cnum}\,\varphi(e^{- (\cos \theta+\mi \sin \theta)t} x +\sqrt{1-e^{-2t\cos \theta}}y)\,\dif \gamma(y),\quad \forall \varphi\in C^0_{\uparrow}(R^2).
\end{equation*}
Then $T_t^2\varphi(x)=P_t\varphi(x)=E_x[\varphi(Z_t)]$ for each $\varphi\in C^0_{\uparrow}(\Rnum^2)$.

Second, using (\ref{cp}), it follows from Ito's lemma and Theorem~\ref{pop27} that for each $\varphi\in C^2_{\uparrow}(\Rnum^2)$,
\begin{align*}
  \varphi(Z_t)&=\varphi(x)+\int_{0}^t\frac{\partial}{\partial z}\varphi (Z_s) \,\dif Z_s+ \int_{0}^t\frac{\partial}{\partial \bar{z}}\varphi (Z_s) \,\dif \bar{Z}_s+  \int_0^t \frac{\partial^2}{\partial z\partial \bar{z}}(Z_s)\dif\innp{Z,\bar{Z} }_s\\
  &=\varphi(x)+\int_0^t \mathcal{L}_{\theta}\varphi (Z_s) \,\dif s- \sqrt{2\cos\theta}(\int_0^t \frac{\partial\varphi}{\partial z}(Z_s)\dif \zeta_s+ \int_0^t \frac{\partial\varphi}{\partial \bar{z}}(Z_s)\dif \bar{\zeta}_s).
\end{align*}
Then
\begin{align*}
  T_t^2\varphi(x) & =E_x[\varphi(Z_t)]=\varphi(x)+ E_x[\int_0^t \mathcal{L}_{\theta}\varphi (Z_s) \,\dif s] \\
    &=\varphi(x)+\int_0^t E_x[ \mathcal{L}_{\theta}\varphi (Z_s)] \,\dif s\quad (\text{by Fubini Theorem})\\
    &=\varphi(x)+\int_0^t T_s^2 \mathcal{L}_{\theta}\varphi (x) \,\dif s,
\end{align*}
 and
  \begin{align*}
    \mathcal{A}_2\varphi &=\lim_{t\downarrow 0}\frac{T_t^2\varphi -\varphi}{t}=\lim_{t\downarrow 0}\frac{1}{t} \int_0^t T_s^2 \mathcal{L}_{\theta}\varphi (x) \,\dif s \\
    &=\mathcal{L}_{\theta}\varphi \quad (\text{in} \quad L^2(\gamma)),
  \end{align*}where to get the last equality we use the continuity of $t \to T^2_t \varphi$ for any $\varphi\in L^2(\gamma)$ (see \cite[Corollary 2.3]{Pazy} or part (a) of \cite[Theorem 2.4]{Pazy}). Therefore, $ \mathcal{A}_2=\mathcal{L}_{\theta}$ on $C^2_{\uparrow}(\Rnum^2)$,

Third, since $\mathrm{graph}(\mathcal{L}_{\theta}|C^2_{\uparrow}(\Rnum^2)) $ is dense in  $\mathrm{graph}(\mathcal{L}_{\theta}) $ (see Proposition~\ref{pp26}) and $\mathcal{A}_2$ is closed, we have that $\mathcal{L}_{\theta}\subseteq \mathcal{A}_2$.
It follows from Proposition~\ref{pop25} and Theorem~\ref{lm212} that both $\mathcal{L}_{\theta}$ and $\mathcal{A}_2$ are normal operators. Since $\mathcal{L}_{\theta}$ is maximally normal (see \cite[Theorem~13.32]{ru}),  we have that $ \mathcal{A}_2=\mathcal{L}_{\theta}$. {\hfill\large{$\Box$}}

\begin{cor}\label{cor4_8}
   For any $\theta\in(-\frac{\pi}{2},\,\frac{\pi}{2})$, $\mathcal{L}_{\theta}\subseteq\mathcal{A}_1$ (i.e., $\mathcal{A}_1$ is an extension of $  \mathcal{L}_{\theta}$ to $L^1(\gamma)$) and
\begin{equation}\label{ddom}
     \mathcal{D}(\mathcal{L}_{\theta})=\set{\phi\in D(\mathcal{A}_1)\cap L^2(\gamma):\,\mathcal{A}_1\phi\in L^2(\gamma) }.
\end{equation}
\end{cor}

\begin{proof} The proof is similar to the real case \cite[p19]{bd}. In detail,
 $ \mathcal{D}(\mathcal{L}_{\theta})= \mathcal{D}(\mathcal{A}_2)\subset \set{\phi\in D(\mathcal{A}_1)\cap L^2(\gamma):\,\mathcal{A}_1\phi\in L^2(\gamma) } $ is trivial. Now suppose that $\phi\in D(\mathcal{A}_1)\cap L^2(\gamma)$ and $\mathcal{A}_1\phi\in L^2(\gamma)$, then as $t\to 0$,
\begin{align*}
 \frac{ T^2_t\phi-\phi }{t} &= \frac{T^1_t\phi-\phi}{t} \\
&=\frac{1}{t}\int_0^{t}T^1_s\mathcal{A}_1 \phi\dif s\quad (\text{by the semigroup equation})\\
 &=\frac{1}{t}\int_0^{t}T^2_s\mathcal{A}_1 \phi\dif s\to \mathcal{A}_1 \phi \quad (\text{in} \quad L^2(\gamma)),
\end{align*}
where  to get the last equality we use again the continuity of $t \to T^2_t \varphi$. Thus $\phi\in  \mathcal{D}(\mathcal{A}_2)= \mathcal{D}(\mathcal{L}_{\theta})$ and (\ref{ddom}) holds.
\end{proof}
\noindent{\it Proof of 4),5) of Theorem~\ref{main_thm}.\,}
 Since $F\in C^2_{\uparrow}(\Cnum^{n})$ and $\vec{\phi}=(\phi_1,\dots,\phi_n)\in \mathcal{D}^n$, then $F\circ\vec{\phi}\in C^2_{\uparrow}(\Rnum^2)$.
   By the complex version of the chain rule \cite[p27]{st}, it follows from Theorem~\ref{pop27} that
  \begin{align}
        \mathcal{L}_{\theta}(F\circ\vec{\phi})&=   [ 4\cos\theta \frac{\partial^2}{\partial z\partial \bar{z}}-e^{\mi\theta} z \frac{\partial}{\partial z}-e^{-\mi\theta}\bar{z} \frac{\partial}{\partial \bar{z}}] (F\circ\vec{\phi})\nonumber\\
      &= -e^{\mi\theta} z \sum_{i=1}^n \frac{\partial}{\partial z} \phi_i\frac{\partial F}{\partial z_i}\circ\vec{\phi}+ \frac{\partial}{\partial z} \bar{\phi}_i\frac{\partial F}{\partial \bar{z}_i}\circ\vec{\phi}  \nonumber\\
      &-e^{-\mi\theta}\bar{z} \sum_{i=1}^n \frac{\partial}{\partial \bar{z}} \phi_i\frac{\partial F}{\partial z_i}\circ\vec{\phi}+ \frac{\partial}{\partial\bar{ z}} \bar{\phi}_i \frac{\partial F}{\partial \bar{z}_i}\circ\vec{\phi} \nonumber\\
      &+ 4\cos\theta  \sum_{i=1}^n  \frac{\partial^2 \phi_i}{\partial  z\partial \bar{z}}\frac{\partial F}{\partial z_i}\circ\vec{\phi}+ \frac{\partial^2 \bar{\phi}_i }{ \partial z\partial\bar{ z}}\frac{\partial F}{\partial \bar{z}_i}\circ\vec{\phi}\nonumber\\
       &+ 4\cos\theta  \sum_{i,j=1}^n\frac{\partial}{\partial z} \phi_i( \frac{\partial}{\partial \bar{z}} \phi_j\frac{\partial^2 F}{\partial z_i\partial z_j}\circ\vec{\phi}+ \frac{\partial}{\partial\bar{ z}} \bar{\phi}_j \frac{\partial^2 F}{\partial {z}_i\partial \bar{z}_j}\circ\vec{\phi})\nonumber\\
       &+ 4\cos\theta  \sum_{i,j=1}^n\frac{\partial}{\partial z}\bar{ \phi}_i (\frac{\partial}{\partial \bar{z}} \phi_j\frac{\partial^2 F}{\partial \bar{z}_i\partial z_j}\circ\vec{\phi}+\frac{\partial}{\partial\bar{ z}} \bar{\phi}_j \frac{\partial^2 F}{\partial \bar{z}_i\partial \bar{z}_j}\circ\vec{\phi})\nonumber\\
       &=4\cos\theta  \sum_{i,j=1}^n\frac{\partial}{\partial z} \phi_i \frac{\partial}{\partial \bar{z}} \phi_j\frac{\partial^2 F}{\partial z_i\partial z_j}\circ\vec{\phi}+ \frac{\partial}{\partial  z } \bar{\phi}_i\frac{\partial}{\partial\bar{ z}} \bar{\phi}_j \frac{\partial^2 F}{\partial \bar{z}_i\partial \bar{z}_j}\circ\vec{\phi}\nonumber\\
       &+4\cos\theta  \sum_{i,j=1}^n(\frac{\partial}{\partial z} \phi_i \frac{\partial}{\partial \bar{z}} \bar{\phi}_j + \frac{\partial}{\partial\bar{ z}} \phi_i \frac{\partial}{\partial z} \bar{\phi}_j ) \frac{\partial^2 F}{\partial z_i\partial \bar{z}_j}\circ\vec{\phi}\nonumber\\
       &+\sum_{i=1}^n \mathcal{L}_{\theta}\phi_i\frac{\partial F}{\partial z_i}\circ\vec{\phi}+ \mathcal{L}_{\theta}\bar{\phi}_i\frac{\partial F}{\partial \bar{z}_i}\circ\vec{\phi}.\label{ltheta2}
   \end{align}
   Taking $F(z_1,z_2)=z_1 \overline{z_2}$ in the above equation, we have that
\begin{align}
\bilf{\phi,\,\psi}&=\frac{1}{ 2\cos\theta  }[\mathcal{L}_{\theta}(\phi\bar{\psi })-\bar{\psi }\mathcal{L}_{\theta}(\phi)- \phi\mathcal{L}_{\theta}(\bar{\psi })
]\nonumber\\
&=2[\frac{\partial \phi}{\partial z} \frac{\partial \bar{\psi}}{\partial \bar{z}} + \frac{\partial \phi}{\partial\bar{ z}} \frac{\partial\bar{\psi}}{\partial z}] =2[\frac{\partial \phi}{\partial z}\overline{ \frac{\partial{\psi}}{\partial{z}}} + \frac{\partial \phi}{\partial\bar{ z}} \overline{\frac{\partial\psi}{\partial \bar{z}}}] .\label{bilinear}
\end{align}
Clearly, $\bilf{\phi,\,\psi}$ is a non-negative definite bilinear form on the field $\Cnum$.  Substituting (\ref{bilinear}) into (\ref{ltheta2}), we show (\ref{ltheta}).
{\hfill\large{$\Box$}}
\\

\noindent{\it Proof of Theorem~\ref{main_thm}.\,} By Proposition~\ref{cld1}, the operator $\mathcal{L}_{\theta}$ is closed. 4)-5) of Theorem~\ref{main_thm} have been shown before.
 1)-3) and 6) of Theorem~\ref{main_thm} are shown in Proposition~\ref{pop25}-\ref{pp26} and Corollary~\ref{cor4_8} respectively.
{\hfill\large{$\Box$}}
\section{Appendix}\label{app}
To be self-contained, we list the necessary results of functions slowly
increasing at infinity. Some results which can not be found in textbooks will be shown shortly here.
In this section all functions will be complex-valued and defined on $\Rnum^n$.
\begin{dfn}\label{dfn5_1}
Denote by $ C^{\infty}_{c}(\Rnum^n)$ the space of smooth and compactly supported functions on $\Rnum^n$ \cite[p5]{bn}). Denote by $S(\Rnum^n) $ the space of $C^{\infty}$ functions rapidly decreasing at infinity \cite[p105]{bn}.
We say that a continuous function $f(x)$ is slowly increasing at infinity if there exists an integer $k$ such that
$(1+r^2)^{-\frac{k}{2}}f(x)$ is bounded in $\Rnum^n$ with $r=\abs{x}$ \cite[p110]{bn}. Denote by $C^m_{\uparrow}(\Rnum^n)$ the space of all
functions having slowly increasing at infinity continuous partial derivatives of
order $\le m$.
\end{dfn}

\begin{nott} Denote by $\gamma$ the n-dimensional standard Gaussian measure:
   \begin{equation*}
     \dif \gamma(x)=(2\pi)^{-\frac{n}{2}}\exp\set{-\frac{\abs{x}^2}{2}}\dif x,\quad x\in \Rnum^n.
   \end{equation*}
   Denote the density function by $\rho(x)=\frac{\dif \gamma(x)}{\dif x}$.
\end{nott}

\subsection{ Approximation of identity of $C^m_{\uparrow}(\Rnum^n)$ in $L^q(\gamma)$ }
\begin{nott} Set
 \begin{equation}
 \phi(x) =e^{-\frac{1}{1-x^2}} \mathbf{1}_{\{|x|<1\}},\quad x\in \Rnum^n,
 \end{equation} where $\abs{x}=\sqrt{x_1^2+\cdots+x_n^2}$ and $\mathbf{1}_B$ the characteristic function of set $B$.
Divide this function by its integral over the whole space to get a function $\alpha(x)$ of integral one which is called a mollifier. Next, for every $\epsilon >0$, define \cite[p5]{bn}
\begin{equation}
   \alpha_{\epsilon}(x)=\frac{1}{\epsilon^n}\alpha(\frac{x}{\epsilon}).
\end{equation}
Let $L^1_{loc}(\Rnum^n)$ be the space of locally integrable function on $\Rnum^n$. If $u\in L^1_{loc}(\Rnum^n)$, the function
   \begin{equation}\label{conv}
     u_{\epsilon}(x)=\int_{\Rnum^n} u(x-y)\alpha_{\epsilon}(y)\dif y=\int_{\Rnum^n}\alpha_{\epsilon}(x-y) u(y)\dif y
   \end{equation}
   is said to be the convolution of $u$ and $\alpha_{\epsilon}$ \cite[Definition 1.4]{bn}. It is also denoted by the convolution operator $A_{\epsilon}u=(u\ast\alpha_{\epsilon})(x) $.
\end{nott}


\begin{lem}\label{ll}
Suppose that $f(x)\in C^0_{\uparrow}(\Rnum^n)$. Then $A_{\epsilon}f\in  C^{\infty}_{\uparrow}(\Rnum^n)$ (the space of $C^{\infty}$ functions slowly increasing at infinity) and for any $q\ge 1$ and any $k\in \Nnum^n$, $\lim\limits_{\epsilon \to 0} {x}^k A_{\epsilon}f ={x}^k f$ in $L^q(\gamma)$.
\end{lem}
\begin{proof}
 First, for any $\epsilon>0$, since $\alpha_{\epsilon}\in C^{\infty}_{c}(\Rnum^n)\subset S(\Rnum^n) $ and $f(x)\in C^0_{\uparrow}(\Rnum^n)\subset S'(\Rnum^n)$ (tempered
distributions, see Example 4 in \cite[110]{bn}), it follows from Theorem~4.9 of \cite[p133]{bn} that $A_{\epsilon}f=f\ast \alpha_{\epsilon} \in  C^{\infty}_{\uparrow}(\Rnum^n)$.

Second, since any polynomial $P(x),\,x\in \Rnum^n,$ is in $L^q(\gamma)$, we have $f,\,A_{\epsilon}f\in L^q(\gamma)$. Thus,
 $ \norm{A_{\epsilon}f-f }^q_q =\lim_{n\to\infty}\int_{\mathrm{B}_a}\abs{A_{\epsilon}f-f}^q\,\dif \gamma(x)$ where $\mathrm{B}_a=\set{x\in\Rnum^n,\,\abs{x}\le a}$.

Finally, given $\sigma>0$, there exists $\mathrm{B}_a$ such that
\begin{equation}\label{aa}
\norm{A_{\epsilon}f-f }^q_q\le \int_{\mathrm{B}_a}\abs{A_{\epsilon}f-f}^q\,\dif \gamma(x)+\frac{\sigma}{2}.
\end{equation}
 Note that
\begin{eqnarray}\label{bb}
          \int_{\mathrm{B}_a}\abs{A_{\epsilon}f-f}^q\,\dif \gamma(x) \le \sup_{x\in \mathrm{B}_a}\abs{A_{\epsilon}f-f}^q\gamma(K_n) \le  \sup_{x\in \mathrm{B}_a}\abs{A_{\epsilon}f-f}^q.
\end{eqnarray}
It follows from \cite[Theorem 1.1]{bn} that $A_{\epsilon}f\to f$ uniformly on $\mathrm{B}_a$ as $\epsilon\to 0$. Thus there exists $\epsilon_0>0$ such that $\sup_{x\in \mathrm{B}_a}\abs{A_{\epsilon}f-f}\le (\frac{\sigma}{2})^{\frac{1}{q}}$ for any $0<\epsilon<\epsilon_0 $. Together with (\ref{aa}) and (\ref{bb}), we have that $\norm{A_{\epsilon}f-f }^q_q\le \sigma$, which proves that $A_{\epsilon}f \to f$ in $L^q(\gamma)$, as $\epsilon \to 0$.

Similar to the above proof, it follows that for any $k\in \Nnum^n$, $\lim\limits_{\epsilon \to 0} {x}^k A_{\epsilon}f ={x}^k f$ in $L^q(\gamma)$.
\end{proof}

\begin{cor}\label{cc1}
   Suppose that $f(x)\in C^m_{\uparrow}(\Rnum^n)$. Then for any $p,k\in \Nnum^n$ such that $\abs{p}\le m$, ${x}^k\partial^p( A_{\epsilon}f)\to {x}^k\partial^p f $ in $L^q(\gamma)$, as $\epsilon \to 0$.
\end{cor}
\begin{proof}
First, if $f(x)\in C^m_{\uparrow}(\Rnum^n)$ then $\partial^p f \in C^0_{\uparrow}(\Rnum^n)$ for any $p\in \Nnum^n$ such that $\abs{p}\le m$. It follows from Lemma~\ref{ll} that $ x^kA_{\epsilon} ( \partial^p f) \to x^k\partial^p f$ in $L^q(\gamma)$ for any $k\in \Nnum^n$.

Second, for any $p\in \Nnum^n$, if $u,\, \partial^p u\in L^1_{loc}(\Rnum^n)$ then $
    \partial^p( A_{\epsilon}u)= A_{\epsilon} ( \partial^p u)$.

Finally, since $C^0_{\uparrow}(\Rnum^n)\subset L^1_{loc}(\Rnum^n) $, we have that ${x}^k\partial^p( A_{\epsilon}f)\to {x}^k\partial^p f $ in $L^q(\gamma)$.
\end{proof}

\begin{nott}\label{ntt1}
   Let $a\in\Rnum^{+}$ and denote by $\mathrm{B}_{a+1}$ and $\mathrm{B}_{a}$ concentric balls of radius $a+1$ and $a$, respectively. It follows from Corollary~3 of \cite[p9]{bn} that there exists a so-called (smooth) cutoff function  $\beta_a(x)\in C^{\infty}_{c}(\Rnum^n)$ such that: (i) $0\le \beta_a\le 1$ and $supp{\beta_a}\subset \mathrm{B}_{a+1}$, (ii) $\beta_a(x)=1$ on $\mathrm{B}_{a} $, (iii) for all $p\in \Nnum^n$, $\sup_{x\in \Rnum}\abs{\partial^p \beta_a}\le c(n,p)$.
\end{nott}
\begin{lem}\label{ll22}
Let the cutoff function $\beta_a$ prevail. Suppose that $g\in C^{\infty}_{\uparrow}(\Rnum^n)$ and set $g_a= g\beta_a$. Then $g_a\in C^{\infty}_c(\Rnum^n)$, and for any $k,p\in \Nnum^n$, $\lim\limits_{a \to \infty} {x}^k \partial^p g_a ={x}^k \partial^p g$ in $L^q(\gamma)$ for any $q\ge 1$.
\end{lem}
\begin{proof}
The Lebniz's rule implies that
   \begin{align*}
     \partial^p g_a   & =\sum_{l\le p}{p \choose l} \partial^l \beta_a \partial^{p-l} g.
   \end{align*}
Denote $\mathrm{G}_a=\Rnum^n - \mathrm{B}_{a}$, it follows from (i)-(iii) of Notation~\ref{ntt1} that
   \begin{align}\label{aaa}
    \abs{ \partial^p g_a -\partial^p g  }     & = \mathbf{1}_{\mathrm{G}_a}\abs{ (\beta_a-1)\partial^p g + \sum_{0< l\le p}{p \choose l} \partial^l \beta_a \partial^{p-l} g }\nonumber \\
    & \le  c\times \mathbf{1}_{\mathrm{G}_a} \sum_{l\le p}\abs{  \partial^{p-l} g},
   \end{align} where $c=\max\limits_{0<l\le p}{c(n,p-l)}\times \max{{p \choose l}}$.

   Since $g\in C^{\infty}_{\uparrow}(\Rnum^n)$, we have $h(x):={x}^k \sum_{l\le p} \abs{ \partial^{p-l} g} \in C^{\infty}_{\uparrow}(\Rnum^n)\subset L^q(\gamma)$. Therefore, $h\mathbf{1}_{\mathrm{G}_a}\to 0$ in $L^q(\gamma)$ as $a\to \infty$. Together with (\ref{aaa}), we have that $\lim\limits_{a \to \infty} {x}^k \partial^p g_a ={x}^k \partial^p g$ in $L^q(\gamma)$.
\end{proof}

\begin{prop}\label{pp47}
{\bf (Approximation of identity of $C^m_{\uparrow}(\Rnum^n)$ in $L^q(\gamma)$)}\\
      Suppose that $f(x)\in C^m_{\uparrow}(\Rnum^n)$. Denote
       \begin{equation*}
     B_{\epsilon}f=  \beta_{\frac{1}{\epsilon}}\times A_{\epsilon}f.
\end{equation*}
Then $B_{\epsilon}f\in C_c^{\infty}(\Rnum^n)$, and for $q\ge 1$ and $k, p\in \Nnum^n$ such that $\abs{p}\le m$, ${x}^k\partial^p( B_{\epsilon}f)\to {x}^k\partial^p f $ in $L^q(\gamma)$, as $\epsilon \to 0$.
\end{prop}
\begin{proof}
Lemma~\ref{ll} implies that $A_{\epsilon}f\in C^{\infty}_{\uparrow}(\Rnum^n)$. Then it follows from Lemma~\ref{ll22} that  $B_{\epsilon}f\in C_c^{\infty}(\Rnum^n)$ and $ \norm{{x}^k\partial^p( B_{\epsilon}f)- {x}^k\partial^p (A_{\epsilon}f)}_{q}\to 0$. Corollary~\ref{cc1} implies that $ \norm{ {x}^k\partial^p (A_{\epsilon}f)- {x}^k\partial^p f}_{q}\to 0$. By the triangle inequality, we have
   \begin{equation*}
      \norm{{x}^k\partial^p( B_{\epsilon}f)- {x}^k\partial^p f}_{q}\le \norm{{x}^k\partial^p( B_{\epsilon}f)- {x}^k\partial^p (A_{\epsilon}f)}_{q}+ \norm{ {x}^k\partial^p (A_{\epsilon}f)- {x}^k\partial^p f}_{q}\to 0.
   \end{equation*}
\end{proof}

\subsection{ The N-representation theorem for $S(\Rnum^n) $ in  $L^2(\gamma)$}
Suppose $H_l(x)=\frac{(-1)^l}{\sqrt{l!}} e^{x^2/2}\frac{\dif^l}{\dif x^l}e^{-x^2/2}$ is the $l$-th Hermite polynomial of one variable. It is well known that the set of Hermite polynomials of several variables
\begin{equation}\label{bfh}
  \set{ \mathbf{H}_{m}:=\prod_{k=1}^{n} H_{m_k}(x_k),\quad m=(m_1,\dots,m_n)\in \Nnum^n }
\end{equation}
is an orthonormal basis of  $L^2(\gamma)$. Thus, every function $u\in L^2(\gamma)$ has a unique series expression
\begin{equation}
   u=\sum_{m\in \Nnum^n} a_m \mathbf{H}_{m},
\end{equation} where the coefficients $a_{m}$ are given by
\begin{equation*}
a_m=\int_{\Rnum^n}u(x)\mathbf{H}_{m}(x)\,\dif \gamma(x).
\end{equation*}
\begin{prop}\label{tm57}
   $u\in L^2(\gamma)$ satisfies that $a_m=\int_{\Rnum^n}u(x)\mathbf{H}_{m}(x)\,\dif \gamma(x)$ is rapidly decreasing (i.e., for $r\in \Nnum^n\ge 0$, $a_m=O(m^{-r})$ as $\abs{m}\to \infty$) if and only if $u=f\rho^{-\frac12}$ with $f\in S(\Rnum^n)$.
\end{prop}
\begin{proof} Denote the Hermite functions $\mathcal{H}_{m}(x)=\mathbf{H}_{m}(x)\rho^{\frac12}$, then
$\int_{\Rnum^n}u(x)\mathbf{H}_{m}(x)\,\dif \gamma(x)=\int_{\Rnum^n}f(x)\mathcal{H}_{m}(x)\,\dif x $. The desired conclusion is followed from Theorem~3.5 and Exercise 3 of \cite[pp135]{jon}.
\end{proof}

\begin{rem}\label{rm1}
   Clearly, the smooth and compactly supported function satisfies the above condition. In fact,
   \begin{equation*}
      C_c^{\infty}(\Rnum^n)=\set{u=f\rho^{-\frac12}:\,f\in  C_c^{\infty}(\Rnum^n)} \subset \set{ u=f\rho^{-\frac12}:\,f\in S(\Rnum^n)}.
   \end{equation*}
\end{rem}
\begin{prop}\label{tm410}
 If $u\in L^2(\gamma)$ satisfies that $a_m=\int_{\Rnum^n}u(x)\mathbf{H}_{m}(x)\,\dif \gamma(x)$ is rapidly decreasing, then the Hermite expansion $u(x)=\sum_{m\in \Nnum^n}a_m \mathbf{H}_m(x)$ satisfies that
   \begin{equation*}
     \norm {x^k\partial^p (u-u_l)}_{L^2(\gamma)} \to 0 {\text{ as $l\to \infty$,\quad}} \forall k,\,p\in\Nnum^n,
   \end{equation*}
   where $u_l=\sum_{\abs{m}\le l}a_m \mathbf{H}_m(x)$.
\end{prop}
\begin{proof} Proposition~\ref{tm57} implies that $ S(\Rnum^n)\ni f=\sum_m a_m \mathcal{H}_m(x)$. Denote $f_l=u_l\rho^{\frac12}$, then it follows from the N-representation theorem for $S(\Rnum^n) $ (see Theorem~V.13 of \cite[p143]{rs}) that $f_l\to f$ in  $S(\Rnum^n)$ which means that $ \norm{x^m\partial^i (f-f_l)}_{L^2(\dif x)}\to 0\quad \forall m,\,i\in\Nnum^n$ as $l\to \infty$.

The Lebniz's rule implies that there exists a constant $c>0$ such that $$\norm {x^k\partial^p (u-u_l)}_{L^2(\gamma)}\le c \sum_{m\le k+p,\,i\le p} \norm{x^m\partial^i (f-f_l)}_{L^2(\dif x)}.$$ Thus the desired conclusion follows.
\end{proof}



\end{document}